\documentclass[reqno]{amsart}
\usepackage[utf8]{inputenc}
\usepackage[T2A]{fontenc}
\usepackage[english]{babel}
\makeatletter
\def\@settitle{\begin{center}%
	\baselineskip14\p@\relax
	\bfseries
	\@title
	\end{center}%
}
\makeatother

\theoremstyle{theorem}
\newtheorem{theo}{Theorem}[section]

\newtheorem{prop}[theo]{Proposition}
\newtheorem{problem}[theo]{Problem}
\newtheorem{conj}[theo]{Hypothesis}
\newtheorem{conclusion}[theo]{Conclusion}

\theoremstyle{definition}
\newtheorem{defi}[theo]{Definition}
\newtheorem{example}[theo]{Example}

\theoremstyle{remark}
\newtheorem{remark}[theo]{Remark}

\usepackage{amsthm, amsmath, amssymb}
\usepackage{comment}
\usepackage{mathrsfs}
\usepackage{graphicx}
\usepackage{tikz}

\newcommand*\circled[1]{\tikz[baseline=(char.base)]{\node[shape=circle,draw,inner sep=2pt] (char) {#1};}}

\usepackage{multirow}
\usepackage{wrapfig}
\usepackage{caption}
\usepackage{subcaption}
\usepackage{varioref}
\usepackage{xy}
\xyoption{all}
\usepackage[enableskew]{youngtab}
\usepackage{multirow}
\usepackage{comment}
\usepackage{cite}
\usepackage{cleveref}
\usepackage{ifpdf}

\ifpdf
	\makeatletter
		\g@addto@macro\Gin@extensions{,.eps}
	\makeatother
\fi
\title[]{Sharygin triangles and elliptic curves}

\author{
	I.\,V.\,Netay,
	A.\,V.\,Savvateev%
}

\thanks{The research of I.\,V.\,Netay (sections~3,4, appendix~A) was carried out at the IITP RAS at the expense of the Russian Foundation for Sciences (project № 14-50-00150).}
\thanks{A.\,V.\,Savvateev wishes to acknowledge the support of the Ministry of Education and Science of the Russian Federation, grant No. 14.U04.31.0002, administered through the NES CSDSI}

\address{Institute for Information Transmission Problems, RAS}
\address{National Research University Higher School of Economics, Russian Federation}
\address{New Economic School, Moscow Institute of Physics and Technology, and Dmitry Pozharsky University}

\begin{document}
\maketitle


\begin{abstract}
	The paper is devoted to the description of family of scalene triangles for which the triangle formed by the intersection points of bisectors with opposite sides is isosceles.
	We call them Sharygin triangles.
	It turns out that they are parametrized by an open subset of an elliptic curve.
	Also we prove that there are infinitely many non-similar integer Sharygin triangles.
\end{abstract}

\newcounter{p}
\newcounter{d}

\section{Introduction}
	The following problem had been stated in the Kvant journal~(see\cite[page~36]{Kv83}) and had been included into the famous problem book on planimetry~(see~\cite[page 55, problem 158]{Sh}).
	\begin{problem}
		It is known that, for a given triangle, the points where the bisectors meet opposite sides form an isosceles triangle.
		Does it imply that the given triangle is isosceles?
	\end{problem}
	The answer is negative.
	Sharygin writes:
	{\it "Unfortunately, the author had not constructed any explicit example of such a triangle (had not provided a triple of side lengths or a triple of angles) with so exotic property.
	Maybe, the readers can construct an explicit example."}

	For a given triangle, we call the triangle formed by the intersection points of the bisectors with opposite sides the {\it bisectral triangle} (on the~Figure~\ref{fig:triangle1} triangle $A'B'C'$ is bisectral for~$ABC$).

	\begin{defi}
		We call a triangle a {\it Sharygin triangle} if it is scalene but its bisectral triangle is isosceles.
	\end{defi}

	This work is completely devoted to the detailed study of Sharygin triangles.

	A great enthusiast of school mathematical contests Sergei Markelov told us that, amazingly, a Sharygin triangle can be constructed if we take a side of the right heptagon and two different adjacent diagonals (see the proof in Section~\ref{sec:ex}).

	It turns out that any Sharygin triangle has an obtuse angle (it is proved in~\cite{Sh}).
	Moreover, if~$x$ denotes its cosine, then~$-1<4x<\sqrt{17}-5$.
	This implies that the angle measure is between ${\approx} 102.663^\circ$ and~${\approx} 104.478^\circ$.
	In the example arising from the right heptagon we get the obtuse angle~$\frac{8\pi}{7} \approx 102.857^\circ$.
	In respect that the range of suitable angles is very small, this example is totally amazing and surprising.
	Consequently, the following question arises naturally: are there other examples of right polygons such that three of its vertices form a Sharygin triangle?
	Study of this problem has led us to some beautiful formulas, but has not led yet to new examples.

	Of course, an integer Sharygin triangle would be a kind of "triumph" in the problem to construct explicit examples of Sharygin triangles.

	To solve this problem S.\,Markelov started a computation running over all triangles with side lengths not exceeding million.
	No such luck, after two months (in~1996 or around) the computer answered that there are no examples.
	Nevertheless, Sergei had not calmed down.
	Evidently, something suggested him the right answer.

	Consider a triangle~$ABC$.
	Let us denote by $AA'$, $BB'$, and $CC'$ its bisectors and by \par
	$
		\hfill\hfill a=BC, \hfill b=AC, \hfill c=AB \hfill\hfill
	$ \par\noindent
	its side lengths (see~Figure~\ref{fig:triangle1}).
	Sergei considered the replacement
	\[
		\begin{cases}
			a = y+z,\\
			b = x+z,\\
			c = x+y.
		\end{cases}
	\]
	This replacement is well known in planimetry.
	For $a,b,c$ being the side lengths of a triangle, $x,y,z$ are distances from vertices to the points where the incircle meets the adjacent sides~(see~Figure~\ref{fig:triangle2}).

	\begin{figure}
		\centering
		\begin{subfigure}[b]{.45\textwidth}
			\includegraphics[width=\textwidth]{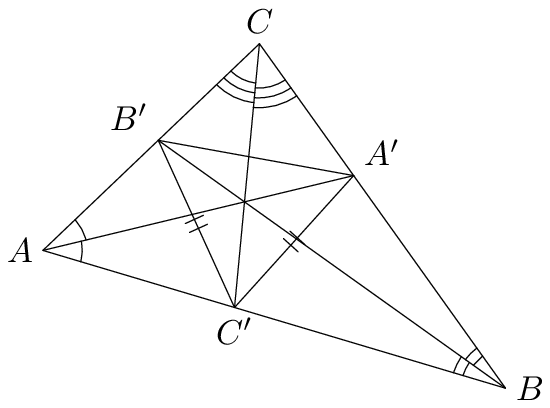}
			\caption{}
			\label{fig:triangle1}
		\end{subfigure}
		~
		\begin{subfigure}[b]{.45\textwidth}
			\includegraphics[width=\textwidth]{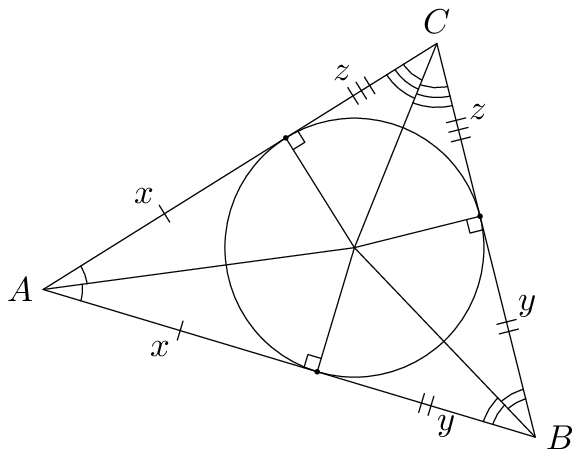}
			\caption{}
			\label{fig:triangle2}
		\end{subfigure}
		\caption{}
		\label{fig:tr}
	\end{figure}

	Sergei had rewritten the equation in terms of~$x,y,z$:
	\[
		4z^3 + 6xyz - 3xy(x+y) + 5z(x^2+y^2) + 9z^2(x+y) = 0.
	\]
	It is enough that~$x,y,z>0$ for~$a,b,c$ to satisfy the triangle inequalities.
	We see that the condition~$A'C'=B'C'$ becomes a cubic equation on~$x,y,z$ without the monomials~$x^3$ and~$y^3$.
	This means that the corresponding projective curve~$\mathcal{E}$ on the projective plane with coordinates~$(x:y:z)$ contains the points~$(1:0:0)$ and~$(0:1:0)$.
	Sergei divided the equation by~$z^3$ passing to the affine chart~$\{z\ne 0\}$ with the coordinates~$\widetilde{x}=\frac{x}{z}$ and~$\widetilde{y}=\frac{y}{z}$, and guessed that the point $(\widetilde{x},\widetilde{y})=(1,-3)$ lies on the curve.

	The equation of the curve~$\mathcal{E}$ is quadratic in~$\widetilde{x}$ and~$\widetilde{y}$.
	On the next step Sergei reopened the addition law of the points on an elliptic curve.
	Namely, he guessed a rational point and started to draw vertical and horizontal lines through the points that are already constructed.
	If we have a point~$Q$ on~$\mathcal{E}$, then the vertical line intersects the curve in two points (actually, in three points, where the third point is on infinity).
	The equation on these two points has rational coefficients, and we know one solution that is rational.
	Therefore, the other one is also rational.
	If we make after the same construction with the horizontal line through the new point and obtain the point~$Q'$, then this construction of~$Q'$ from~$Q$ corresponds to the addition of some point~$A$ to~$Q$ on the curve~$\mathcal{E}$.
	(We need to define the origin of our elliptic curve to determine the point~$A=Q'-Q$ that does not depend on the choice of~$Q$.)

	Iterating this algorithm (passing from~$Q$ to~$Q+A$), Sergei obtained on the fourth step a triangle, i.\,e.\,a~point with~$\widetilde{x},\widetilde{y}>0$.
	Replacing the coordinates back to~$a,b,c$, Sergei obtained $(a,b,c)=(18800081, 1481089, 19214131)$.
	It is not surprising that the computer program running over values before million gives nothing, while the presumably first integer triangle has so grandiose side lengths!
	Sergei would have found this triangle, if he started the program running the side lengths not exceeding billion on a more powerful computer some time after.
	Fortunately, the modern computers are not so powerful.
	This led us to the addition law on the elliptic curve and to some more complicated theory.
	As a result, we have proved that there are infinitely many non-similar integer Sharygin triangles.
	In this way, the school-level problem has led us to a beautiful branch of modern mathematics.

	In this work we consider the question of how to construct each integer Sharygin triangle.
	Integer Sharygin triangles correspond to rational points of the elliptic curve~$\mathcal{E}$ lying in some open subset (defined by the triangle inequalities).
	Therefore, we need to describe the group of rational points on the curve.
	Here we find the torsion subgroup, find the rank and give an element of the free part (it seems to be a generator).
	If this point of infinite order is a generator, then all the integer Sharygin triangles can be constructed from the points that we have found, by the addition of points on~$\mathcal{E}$.

	The authors are grateful to Sergei Markelov for introducing the problem and a series of useful tricks in the first rough solution.
	Also the authors are grateful to V.\,A.\,Voronov, N.\,Tsoy, D.\,V.\,Osipov, S.\,S.\,Galkin, I.\,A.\,Cheltsov, S.\,O.\,Gorchinsky and other participants of the conference <<Magadan Algebraic Geometry>> for useful discussions and advices.
	We are grateful to N.\,N.\,Osipov for the proof of Proposition~\ref{ordA=inf simple}.

	Lastly, we should explain elementary character of the exposition.
	The point was to introduce wide public~--- students, non-algebraic mathematicians and alike~--- to the magic of elliptic curves, starting from the absolutely clear, school-level problem.
	That is why we do not restrict ourselves to refer to standard but involved results from the arithmetics and algebraic aspects of elliptic curves, instead giving absolutely elementary proofs to most of our statements.
	We wish that students would find their speciality in studying elliptic curves, after reading this introductory text.

\section{Sharygin triangle arising from right heptagon}
	\label{sec:ex}
	\begin{example}
		\label{ex:7}
		Consider the unit circle~$|z|=1$ on the complex plane~$\mathbb{C}$.
		Set $\zeta = e^{\frac{2\pi i}{7}}$.
		Consider the triangle~$(1,\zeta,\zeta^3)$.
		Obviously, it is scalene.
		Its vertices are placed at the vertices of the right heptagon drawn by dot-and-dash line on~\cref{fig:7}.

		\begin{figure}[ht]
			\begin{center}
				\includegraphics{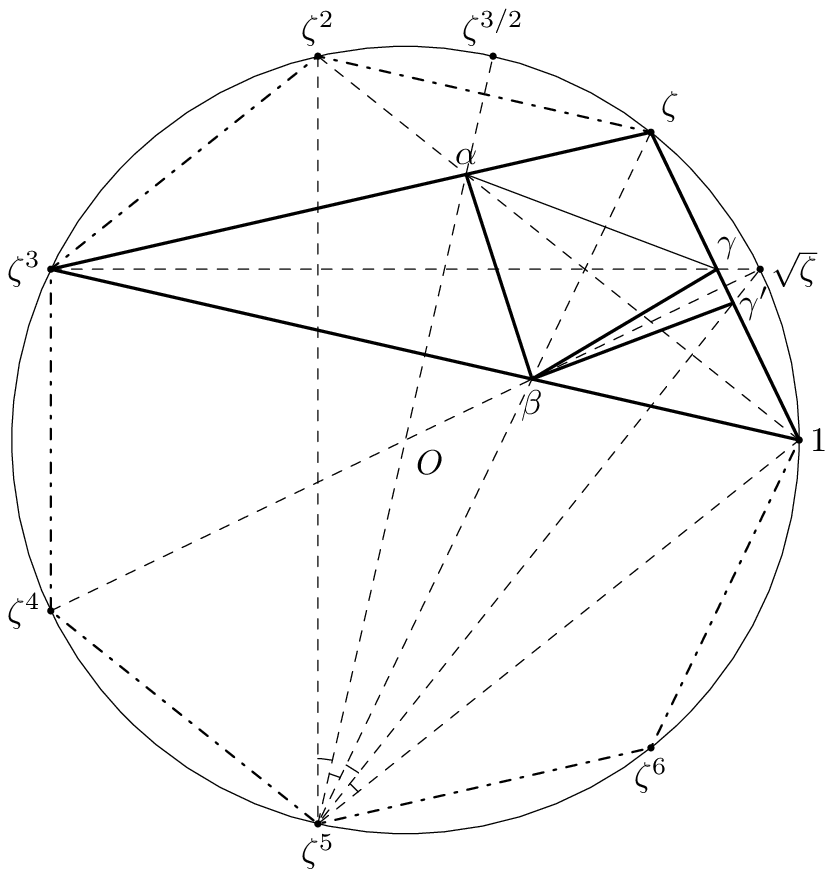}
			\end{center}
			\caption{}
			\label{fig:7}
		\end{figure}

		In Section~\ref{sec:param} we reduce the property of a triangle to be a Sharygin triangle to a cubic relation on its sides.
		Actually, it is enough to substitute side lengths to equation~\eqref{eq:C} to verify that $(1,\zeta,\zeta^3)$ is a Sharygin triangle.
		It is even simpler to substitute its angles~$\alpha=\dfrac{\pi}{7}$ and~$\beta=\dfrac{2\pi}{7}$ into the equivalent equation~\eqref{eq:Ctrig}.
		Let us, however, prove this fact geometrically.

		Denote the points of intersection of bisectors with the opposite sides of~$(1,\zeta,\zeta^3)$ by~$\alpha$, $\beta$ and~$\gamma$ as on the~\cref{fig:7}.
		Let~$\gamma'$ denote the reflection of~$\gamma$ from the line through~$\sqrt{\zeta}$ and~$\zeta^4$.
		Then $|\gamma-\beta|=|\gamma'-\beta|$.
		Lines $(1,\zeta^2)$ and $(\zeta,\zeta^3)$ are symmetric with respect to~$(\zeta^{3/2},\zeta^5)$.
		Therefore~$\alpha$ lies on~$(\zeta^{3/2},\zeta^5)$.
		Lines~$(\zeta^5,\zeta^{3/2})$ and~$(\zeta^5,\sqrt{\zeta})$ are symmetric with respect to~$(\zeta,\zeta^5)$.
		Lines~$(\zeta,\zeta^3)$ and~$(\zeta,1)$ are also symmetric with respect to~$(\zeta,\zeta^5)$.
		Therefore~$|\alpha-\beta|=|\gamma'-\beta|$.
		Finally,~$|\alpha-\beta|=|\gamma-\beta|$.

		Let us prove that the triangle~$(1,\zeta,\zeta^3)$ is not similar to a triangle with integer sides.
		Consider the ratio of two side lengths:
		\[
			\frac{|\zeta^3-\zeta|}{|\zeta-1|} = |\zeta+1| = 2\cos\frac{\pi}{7}.
		\]

		One can verify that the number~$2\cos\frac{\pi}{7}$ is a root of the irreducible polynomial~$z^3-z^2-2z+1$.
		Therefore it is irrational.
	\end{example}

	\begin{conj}
		\label{conj7}
		Suppose that vertices of Sharygin triangle coincide with vertices of some right~polygon with~$n$~sides.
		Then~$n$ is divisible by~$7$, and this triangle is similar to the one described above.
	\end{conj}


	Let us give some ideas about~\cref{conj7}.
		From the law of sines we have
		\[
			a = 2\sin\alpha,\quad b=2\sin\beta,\quad c=2\sin\gamma,
		\]
		where $\alpha$, $\beta$ and~$\gamma$ are opposite to the sides $a$, $b$, $c$ angles.
		We will see below that the condition on~$a,b,c$ to form a Sharygin triangle is a cubic equation~\eqref{eq:C}.
		Substituting this to~\eqref{eq:C} with~$\gamma=\pi-\alpha-\beta$, we obtain
		\[
			\sin\alpha\sin\beta(\sin\alpha+\sin\beta+\sin(\alpha+\beta)-\sin(2\alpha+\beta)-\sin(\alpha+2\beta)-\sin(2\alpha+2\beta)) = 0.
		\]
		This can be easily checked by expansion of the brackets in both equations.
		Skipping $\sin\alpha\sin\beta$, we obtain the equation
		\begin{equation}
			\label{eq:Ctrig}
			\sin\alpha+\sin\beta+\sin(\alpha+\beta) = \sin(\alpha+2\beta) + \sin(2\alpha+\beta) + \sin(2\alpha+2\beta).
		\end{equation}

		Consider complex numbers~$x=\cos\alpha+i\sin\alpha$ and~$y=\cos\beta + i\sin\beta$.
		Here~$0<\alpha,\beta<\frac{\pi}{2}$.
		Equation~\eqref{eq:Ctrig} implies that the imaginary part of the number~$w=x+y+xy-xy^2-xy^2-x^2y^2$ is zero.
		This implies that this number coincides with its conjugate.
		Note that~$\bar{x}=x^{-1}$ and~$\bar{y}=y^{-1}$.
		Therefore
		\[
			w-\bar{w} = \frac{xy-1}{x^2y^2}(1+x+y+x^2y^3+x^3y^2+x^3y^3) = 0.
		\]

		Suppose that the vertices coincide with vertices of right $N$-gon.
		Therefore~$\alpha=\frac{m\pi}{N}$ and~$\beta=\frac{n\pi}{N}$ for some integer~$m,n$.
		Therefore~$x$ and~$y$ are roots of unity of degree~$2N$.
		So it is enough to solve the system
		\[
			\begin{cases}
				1 + x + y + x^2y^3 + x^3y^2 + x^3y^3 = 0,\\
				x^N = y^N = 1
			\end{cases}
		\]
		for some~$x,y\in\mathbb{C}$ such that~$x,y,xy$ have positive real and imaginary parts.
		Numerical computation shows that there are no solutions except primitive roots of degree~$7$ for~$N \leqslant 2000$.
		This suggests us to give~\cref{conj7}.
		But we don not know how to prove it.

\section{Parameterization by an open subset of elliptic curve}
	\label{sec:param}
	Take a triangle~$ABC$.
	Let~$A'B'C'$ be its bisectral triangle. Put \par
	\nopagebreak
	$
		\hfill\hfill a'=B'C',\hfill b'=A'C',\hfill c'=A'B'.\hfil\hfill
	$

	\begin{prop}
		Triples of side lengths of Sharygin triangles~$ABC$ are all triples~$(a,b,c)$ satisfying equation
		\begin{equation}
			\label{eq:C}
			q(a,b,c) = - c^3 - c^2(a+b) + c(a^2+ab+b^2) + (a^3 + a^2b + ab^2 + b^3) = 0
		\end{equation}
		and the triangle inequalities
		\[
			\begin{cases}
				0 < a < b + c,\\
				0 < b < a + c,\\
				0 < c < a + b.\\
			\end{cases}
		\]
	\end{prop}

	\begin{figure}[ht]
		\begin{center}
			\includegraphics[scale=.7]{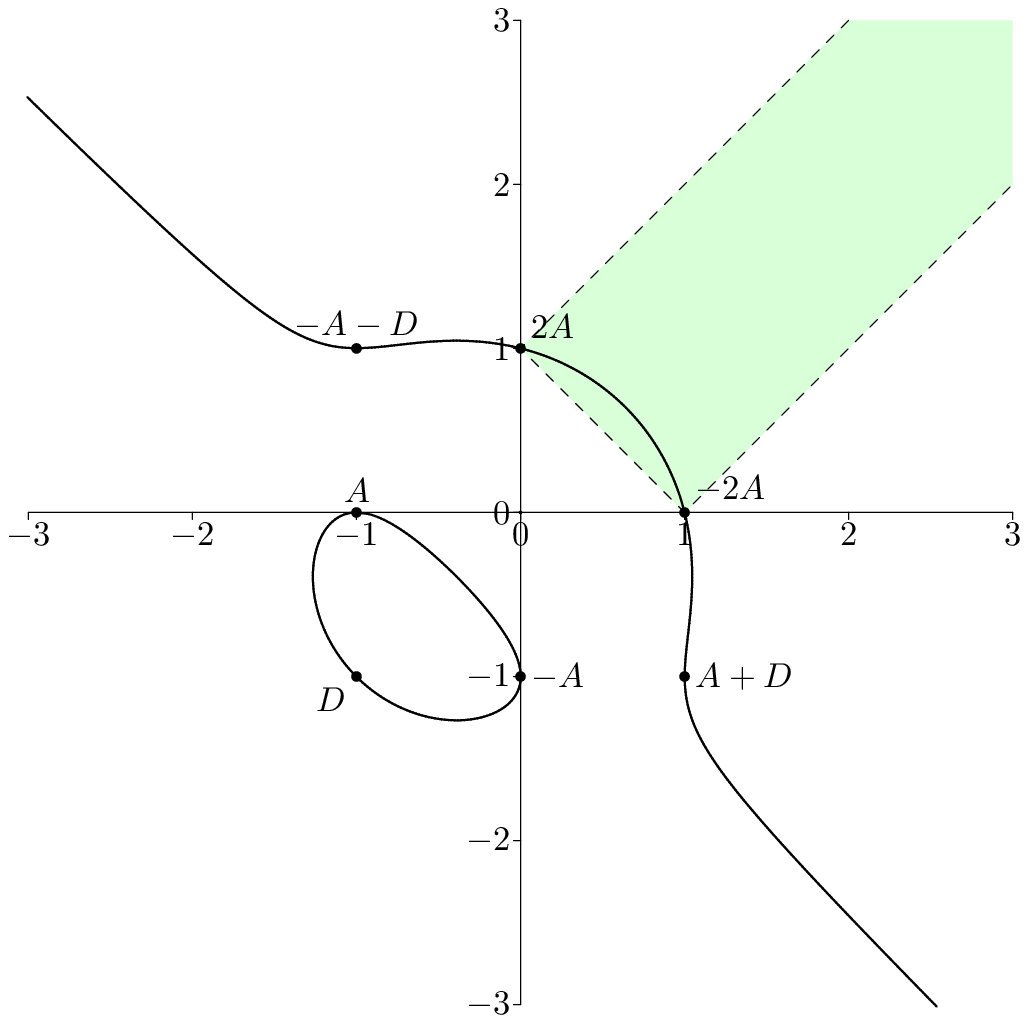}
		\end{center}
		\caption{}
		\label{pict:C}
	\end{figure}
	\begin{proof}
		The law of cosines and the property of bisectors allows us to express~$a',b',c'$.
		We obtain
		\[
			a'-b' = (a-b) \frac{abc}{(a+b)(a+c)^2(b+c)^2}(a^3+a^2b+ab^2+b^3+c(a^2+ab+b^2)-c^2(a+b)-c^3).
		\]
		The first multiplier~$(a-b)$ corresponds to isosceles triangles~$ABC$.
		The second multiplier~$\frac{abc}{(a+b)(a+c)^2(b+c)^2}$ is always positive.
		Therefore, if $a,b,c$ are pairwise different, then the equation~\eqref{eq:C} holds.
		Conversely, if~$a,b,c$ satisfy eq.\,\eqref{eq:C}, then~$a'=b'$.
	\end{proof}
	
	Let us denote by~$\mathcal{E}$ the curve~$q(a,b,c)=0$ in the projective plane with the coordinates~$(a:b:c)$.
	We see on the~\cref{pict:C} the curve~$\mathcal{E}$ in the affine plane~$c=1$ with coordinates~$(a/c,b/c)$ and the domain~$\mathcal{T}$ where the triangle inequalities hold:
	\[\left\{
		\begin{aligned}
			0 < a &< b + 1, \\
			0 < b &< a + 1, \\
			1 &< a + b.\\
		\end{aligned}
	\right.\]

	It is easy to see that the intersection is non-empty.
	For example, the point $(1,0)$ lies on~$\mathcal{E}$ and on the boundary of~$\mathcal{T}$.
	The tangent of~$\mathcal{E}$ at~$(1,0)$ has the equation~$x+y/4=1$.
	Therefore, there are infinitely many real points of~$\mathcal{E}$ in~$\mathcal{T}$.

	This proves that there are infinitely many scalene pairwise non-similar triangles with isosceles bisectral triangles.
	Therefore, we obtain that there are infinitely many Sharygin triangles with real side lengths.
	Below we consider integer Sharygin triangles.

\section{Integer triangles and rational points on elliptic curve}
	\label{sec:ell}
	\subsection{Smoothness and inflexion points}
		We have seen above that Sharygin triangles are parametrized by an open subset of a cubic curve~$\mathcal{E}$ defined by the equation~\eqref{eq:C}.
	\begin{prop}
		The curve $\mathcal{E}$ is an elliptic curve.
	\end{prop}

	\begin{proof}
		It is sufficient to check that the curve~$\mathcal{E}$ is smooth.
		The system~$\frac{\partial q}{\partial a}=\frac{\partial q}{\partial b}=\frac{\partial q}{\partial c}=0$ has no non-trivial solutions.
		This implies that the curve~$\mathcal{E}$ is smooth and therefore is an elliptic curve.
	\end{proof}

	The fact that an elliptic curve has~$9$ inflexion points is well known (see~\cite[Ch.\,IV, \S2, Ex.\,2.3.g, page\,305]{H77}).
	We want to find some inflexion point and consider it as the origin of the elliptic curve~$\mathcal{E}$, because in this case the addition law is simpler then for other choices of the origin.

	\begin{prop}
		The only inflexion point of~$\mathcal{E}$ defined over~$\mathbb{Q}$ is~$(1:-1:0)$.
	\end{prop}

	\begin{proof}
		The inflexion points are defined as the intersection points of~$\mathcal{E}$ and its Hessian being also a smooth cubic.
		It can be easily verified that there are~$9$ intersection points and that the only point among them defined over~$\mathbb{Q}$ is~$(1:-1:0)$.
		All the others become defined over the extension of~$\mathbb{Q}$ by the irreducible polynomial
		\[
			32 + 115 t + 506 t^2 + 1053 t^3 + 1212 t^4 + 1053 t^5 + 506 t^6 + 115 t^7 + 32 t^8.
		\]
		Therefore, the proposition is proved.
	\end{proof}

	We take the point~$O:=(1:-1:0)$ as the identity element of the elliptic curve~$\mathcal{E}$.
	Then for a point~$A$ with coordinates~$(a:b:c)$ the point~$-A$ has the coordinates~$(b:a:c)$.
	Indeed, it is easy to check that these three points lie on one line and the equation is symmetric under the permutation~$a \leftrightarrow b$.

	\subsection{Torsion subgroup}
	At first, we need to find the Weierstra\ss{} form of~$\mathcal{E}$ to find its torsion subgroup.
	Under the change of coordinates
	\[
		\begin{aligned}
			a \mapsto x+y,\\
			b \mapsto x-y,\\
			c \mapsto 24-4x
		\end{aligned}
	\]
	the equation~$q=0$ transforms to
	\[
		y^2 = x^3 + 5x^2 - 32x.
	\]
	The discriminant equals~$\Delta=2506752 =2^{14}\cdot3^{2}\cdot{17}$.
	The set of points of finite order can be easily described using the following result.

	\begin{theo}[Nagell--Lutz, \cite{N,L}]
		Let~$\mathcal{E}$ be an elliptic curve $y^2=x^3+ax^2+bx+c$ with~$a,b,c\in\mathbb{Z}$.
		If a point~$(x,y)\ne\infty$ is a torsion point of~$\mathcal{E}(\mathbb{Q})$, then
		\begin{itemize}
			\item $x,y\in\mathbb{Z}$,
			\item either~$y=0$, or~$y$ divides~$\Delta=-4a^3c+a^2b^2+18abc-4b^3-27b^2$.
		\end{itemize}
	\end{theo}

	\begin{prop}
		\label{pr:tor}
		The torsion subgroup of~$\mathcal{E}(\mathbb{Q})$ is isomorphic to~$\mathbb{Z}/2\mathbb{Z}$ and consists of the points~$O$ and~$(1:1:-1)$, which we denote hereafter by~$D$.
	\end{prop}

	\begin{proof}
		We can consider all the divisors~$y$ of~$\Delta$ and for each~$y$ find all integer solutions~$x$.
		It can be verified that the unique solution is~$(x,y)=(0,0)$.
		This point has~$(a:b:c)$-coordinates equal to~$(1:1:-1)$.
	\end{proof}

	Let us give some elementary proof of the fact that the point~$A=(1:0:-1)$ has infinite order.
	Applying the torsion group description, we can say that~$A\ne O, D$, but this proof uses the Nagell--Lutz Theorem.
	The proof below does not use it and is almost school-level.

	\begin{prop}
		\label{ordA=inf simple}
		The curve~$\mathcal{E}$ has infinitely many rational points.
		In particular, the point~$A=(1:0:-1)$ has infinite order.
	\end{prop}

	\begin{proof}
		Take a rational point~$P_0=(x_0,y_0)$ on the elliptic curve~$\mathcal{E}$ such that~$y_0\ne 0$.
		Consider its {\it duplication}, i.\,e. the point~$P_1=2P_0=(x_1,y_1)$.
		It can be constructed as the intersection point of the tangent line to~$\mathcal{E}$ at~$P_0$.
		It is easy to check that
		\[
			x_0 = \frac{(x_0^2+32)^2}{4x_0(x_0^2+5x_0-32)} \ne 0.
		\]
		Let $x_0=p_0/q_0$ and $x_1=p_0/q_0$ be the irreducible ratios.
		Then
		\[
			\frac{p_1}{q_1} = \frac{(p_0+32q_0^2)^2}{4p_0q_0(p_0+5p_0q_0-32q_0^2)}.
		\]
		Suppose that~$p_0$ is odd and positive (for example, $P_0=6A$ with~$x_0=p_0/q_0=\frac{121}{16}$).
		Then~$p_1$ is also odd and positive.
		It is easy to see that
		\begin{multline*}
			d = \operatorname{GCD}((p_0^2 + 32q_0^2)^2, 4p_0q_0(p_0^2 + 5p_0q_0 - 32q_0^2)) = \\ =
			\operatorname{GCD} ((p_0^2 + 32q_0^2)^2, p_0^2 + 5p_0q_0 - 32q_0^2) = \\ =
			\operatorname{GCD} (q_0^2(64q_0 - 5p_0)^2, p_0^2 + 5p_0q_0 - 32q_0^2) = \\ =
			\operatorname{GCD} (25p_0^2 - 640p_0q_0 + 4096q_0^2, p_0^2 + 5p_0q_0 - 32q_0^2) = \\ =
			\operatorname{GCD} (153(32q_0-5p_0), p_0^2 + 5p_0q_0 - 32q_0^2) = \\ =
			\operatorname{GCD} (153p_0^2, p_0^2+5p_0q_0-32q_0) \in\{1,3,9,17,51,153\},
		\end{multline*}
		because~$p_0^2$ and~$p_0^2+5p_0q_0-32q_0^2$ are coprime.
		We see that
		\[
			p_1 = \frac{(p_0+32q_0^2)^2}{d} \geqslant
			\frac{(p_0^2+32q_0^2)^2}{153} \geqslant
			\frac{(p_0^2+32)^2}{153} > p_0.
		\]
		Therefore, numerators of $x$ coordinates of points~$P_i$, where $P_{i+1}=2P_i$, increase (in particular, for~$P_0=6A$).
		We conclude that~$A$ has infinite order.
	\end{proof}

	\begin{theo}
		Rational points are dense on the curve $\mathcal{E}$.
		There are infinitely many pairwise non-similar integer Sharygin triangles.
	\end{theo}

	\begin{proof}
	Consider the rational point~$A=(1:0:-1)$ on~$\mathcal{E}$.
	Since~$A\ne O,D$, its order is infinite.
	Firstly, the fact that~$\operatorname{ord}A=\infty$ was proved in other way.
	It was checked that points~$nA$ are pairwise different for~$n=1,\ldots,12$\footnote{Thanks to N.\,Tsoy for this calculation.}.
	From Mazur Theorem it follows that order of any torsion point of any elliptic curve does not exceed~$12$.
	Therefore,~$A$ is not a torsion point.

	Consider the equation~\eqref{eq:C} in three-dimensional affine space.
	It is homogeneous,~i.\,e. if a point~$(a,b,c)\ne 0$ is a solution, then any point of the line~$(\lambda a,\lambda b,\lambda c)$ for any~$\lambda$ is a solution.
	In other words, the set of solutions is a cone.
	Consider the unit sphere~$S=\{a^2+b^2+c^2=1\}$.
	The cone intersects it in a curve~$\widetilde{\mathcal{E}}$ that consists of three ovals (see~Figure~\ref{pict:cone}).
	Under the map of~$S$ into the projective plane~$\mathbb{P}^2=\{(a:b:c)\}$ points of projective plane correspond to pairs of opposite points on~$S$.
	Two of the ovals are opposite on the sphere, and another one is opposite to itself.
	
	\begin{figure}[ht]
		\begin{center}
			\includegraphics[scale=1]{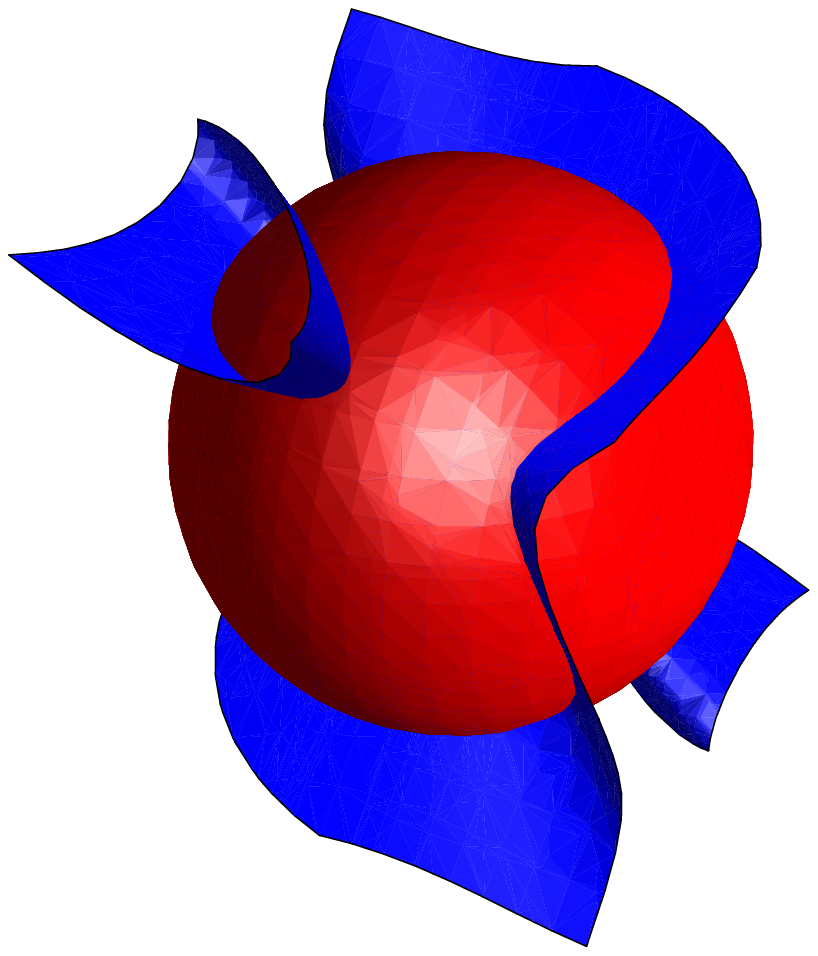}
		\end{center}
		\caption{}
		\label{pict:cone}
	\end{figure}

	For any point~$nA=(a_n:b_n:c_n)$ of~$\mathcal{E}$ denote by~$A_n$ one of two points of intersection of the line~$(\lambda a_n,\lambda b_n, \lambda c_n)$ and the sphere~$S$.
	Since all these lines are different and any pair of them intersects only at the origin, all the points~$A_n$ are pairwise different.
	Therefore,~$\{A_n\}$ is an infinite subset of the compact set~$S\subset\mathbb{R}^3$.
	From Bolzano--Weierstrass theorem it follows that there is a limit point~$\widetilde{L}$ of the set~$\{A_n\}$.
	The curve~$\widetilde{\mathcal{E}}$ is closed, and~$A_n\in\widetilde{\mathcal{E}}$ for any~$n$, therefore~$\widetilde{L}\in\widetilde{\mathcal{E}}$.
	The point~$\widetilde{L}\in\widetilde{\mathcal{E}}$ corresponds to a point~$L\in\mathcal{E}$ that is a limit point of the set~$nA$.

	It is easy to see that the addition~$(X,Y)\mapsto X+Y$ and inversion~$X\mapsto-X$ of points of~$\mathcal{E}$ are continuous operations.
	Denote by~$M$ the point~$(18800081: 1481089: 19214131)$ corresponding to the known Sharygin triangle.
	Let us introduce the function~$f(X,Y)=M+X-Y$ of points of the elliptic curve~$\mathcal{E}$.
	Obviously, it is continuous, and~$f(L,L)=M$.
	From the definition, for any~$\varepsilon>0$ exists~$\delta>0$ such that if~$X,Y\in O_\delta(L)$, then~$f(X,Y)\in O_\varepsilon(M)$.
	We can take such $\varepsilon>0$ that $O_\varepsilon\subset \mathcal{T}$, i.\,e. any point of~$\mathcal{E}\cap O_\varepsilon(M)$ corresponds to a Sharygin triangle.
	For the corresponding~$\delta>0$ there are~$n_\varepsilon A, m_\varepsilon A\in O_\delta(L)$.

	The point~$f(n_\varepsilon A,m_\varepsilon A)$
	\begin{itemize}
		\item corresponds to a Sharygin triangle,
		\item is rational.
	\end{itemize}
	Therefore, we obtain an integer Sharygin triangle in arbitrary small neighborhood of~$M$.
	So we get infinitely many integer Sharygin triangles.

	Now let us prove that rational points are dense on~$\mathcal{E}$.
	We see that topologically~$\mathcal{E}$ is a union of two circles.
	One can see that as a topological group $\mathcal{E}$ is a product~$S^1\oplus\mathbb{Z}/2\mathbb{Z}$, and any of two circles has a rational point (for example,~$O$ and~$D$).
	Therefore the set of points~$\{nA\}\sqcup\{nA+D\}$ consists of rational points and is dense in~$\mathcal{E}$.
	\end{proof}

	\subsection{Rank}
	The {\it rank} of an elliptic curve~$E$ over the field~$\mathbb{Q}$ is defined as the rank of it as an abelian group.
	Unlike the torsion subgroup, there is no known algorithm to calculate the rank of any elliptic curve.
	But some results are known and allow to find ranks of some curves.
	One of them is related to the Hasse--Weil function defined below.
	To give the definition, we need to introduce some notations.

	For the elliptic curve~$E$ defined as~$y^2=x^3+ax+b$ for~$a,b\in\mathbb{Q}$ we can change the variables~$x$ and~$y$ in such a way that the curve~$E$ would be defined as~$y^2=x^3+a'x+b'$ for~$a',b'\in\mathbb{Z}$.
	Indeed, if~$d$ is a common denominator of~$a$ and~$b$, then we can replace~$x\mapsto d^2x$ and~$y\mapsto d^3y$.

	Suppose that the elliptic curve~$E$ is defined as~$y^2=x^3+ax+b$ for~$a,b\in\mathbb{Z}$.
	Then we can consider its reduction modulo~$p$ for each prime number~$p$.
	If $p$ does not divide~$\Delta$, then the group homomorphism arises:
	\[
		E(\mathbb{Q}) \to E(\mathbb{F}_p).
	\]
	For the proof and details see~\cite{RS}.

	\begin{defi}
		\label{def:Np}
		For each prime number~$p$ not dividing~$\Delta$ denote by~$N_p$ the number of points on the curve~$E(\mathbb{F}_p)$,~i.\,e.~the number of pairs~$(x,y)$, where $0\leqslant x,y\leqslant p-1$, such that
		\[
			y^2\equiv x^3+ax+b \mod (p).
		\]
	\end{defi}

	\begin{defi}
		\label{def:HW}
		Define the {\it Hasse-Weil $L$-function} of~$E$, a function in complex variable~$s$, by
		\[
			L(E,s) =
			\prod_{p {\not\,|} \Delta}
			\left(
				1 - \frac{1+p-N_p}{p^s} + \frac{p}{p^{2s}}
			\right)^{-1}
			\times
			\prod_{p \mid \Delta}
			\ell_p(E,s)^{-1},
		\]
		where~$\ell_p(E,s)$ is a certain polynomial in~$p^{-s}$ such that $\ell_p(E,1)\ne 0$ (see~\cite[p.\,196]{Tate}).
	\end{defi}

	From the estimation
	\[
		p + 1 - 2\sqrt{p} \leqslant N_p \leqslant p + 1 + 2\sqrt{p}
	\]
	proved by Hasse (see~\cite{Hasse1,Hasse2}) it follows that~$L(E,s)$ converges absolutely and uniformly on compact subsets of the half-plane~$\{\operatorname{Re}(s)>3/2\}$.
	It was proved by Breuil, Conrad, Diamond, Taylor and Wiles (see\cite{W,TW,BCDT}) that~$L(E,s)$ has an analytic continuation to~$\mathbb{C}$.
	It turns out that the behavior of~$L(E,s)$ at~$s=1$ is related to the rank of~$E$ in the following way.
	Define the {\it analytical rank} $\operatorname{rk}_{an}(E)$ as the order of vanishing of~$L(E,s)$ at~$s=1$.

	\begin{theo}[see~\cite{GZ,Ko1,Ko2}]
		\begin{itemize}
			\item If $\operatorname{rk}_{an}(E)=0$, then~$\operatorname{rk}(E)=0$,
			\item If $\operatorname{rk}_{an}(E)=1$, then~$\operatorname{rk}(E)=1$.
		\end{itemize}
	\end{theo}

	\begin{theo}
		\label{pr:rk}
		The rank of the curve~$\mathcal{E}(\mathbb{Q})$ equals~$1$.
	\end{theo}

		It can be checked by the {\tt pari/gp} computer algebra system that the Hasse-Weil $L$-function has order~$1$ at~$s=1$.
		It has the form
		\[
			L(E,s) = s\cdot0.67728489801666901020123734615355993155... + O(s^2).
		\]
		By~\ref{pr:rk} it implies that the curve has rank~$1$.
		But this method uses some computer algorithm that is not well known.
		We give below a method to check this in a way that does not use complicated computer algorithms and applies more complicated geometrical methods instead.

	The corresponding method is called {\it $2$-descent} and is related to computation of weak Mordell--Weil groups of elliptic curves.
	We will not give full proofs and details of the underlying constructions (they all can be found in~\cite{Silv}), but briefly recall some basic notation of the method and shortly explain why does it work.

	Take an elliptic curve~$E$ defined over a number field~$K$.
	We don not specify the field here, because we should take some extension of~$\mathbb{Q}$ below.
	Denote
	\begin{itemize}
		\item by $E(K)$ the set of $K$-points on~$E$,
		\item by $mE(K)$ the subgroup of points $[m]A$ for all~$A\in E(K)$,
		\item by $E[m] \simeq \left(\frac{\mathbb{Z}}{m\mathbb{Z}}\right)^2$ the subgroup of points of order~$m$ over the algebraic closure~$\bar{K}$ of~$K$,
		\item by $G_{L/K}$ the Galois group of field extension $K\subset L$.
	\end{itemize}
	The main idea is to find an embedding of the group~$E(K)/mE(K)$ to some group and then find the preimages of elements in~$E(K)$.
	If we know the image of~$E(K)/mE(K)$ and the torsion subgroup of~$E(K)$, then we can calculate the rank of~$E$.

	\begin{defi}
		{\it Kummer pairing}
		\[
			\kappa\colon E(K) \times G_{\bar{K}/K} \to E[m]
		\]
		is defined as follows.
		For a point~$P\in E(K)$ there exists a point~$Q\in E(\bar{K})$ such that~$[m]Q=P$.
		Then put~$\kappa(P,\sigma) = Q^\sigma - Q$.
	\end{defi}

	It can be proved that this pairing is well-defined, bilinear, has the left kernel~$mE(K)$ and the right kernel~$G_{\bar{K}/L}$, where~$L$ is the field of definition of all the points in~$[m]^{-1}E(K)\subset E(\bar{K})$.
	Therefore, it defines the perfect pairing
	\[
		\kappa \colon \left[E(K)/m E(K)\right] \times G_{L/K} \to E[m].
	\]
	We can view it as an injective morphism
	\[
		\delta_E \colon E(K)/mE(K) \to \operatorname{Hom}(G_{L/K},E[m]),\quad \delta_E(P)(\sigma) = \kappa(P,\sigma).
	\]
	There is a so called {\it Weil pairing}
	\[
		e_m\colon E[m] \times E[m] \to \mu_m,
	\]
	where~$\mu_m \subset \bar{K}$ is the group of roots of unity of degree~$m$.
	It is bilinear, alternating, Galois-invariant and nondegenerate.
	In particular it implies that if~$E[m]\subset E(K)$, then~$\mu_m\subset K$.
	See details in~\cite[III.8.1]{Silv}.
	Hereafter we assume that~$E[m]\subset E(K)$.

	\begin{remark}
		Actually, this is the moment that requires to extend the field.
		We will apply the method in the case~$m=2$.
		But~$\mathcal{E}[2]$ is not defined over~$\mathbb{Q}$.
		So we need to extend the field to~$\mathbb{Q}(\sqrt{17})$.
		Consideration of other~$m$ does not help.
	\end{remark}

	Hilbert`s Theorem 90 says that each homomorphism $G_{\bar{K}/K}\to\mu_m$ has the form
	\[
		\sigma \to \frac{\beta^\sigma}{\beta}, \quad \beta\in\bar{K}^\times,\quad \beta^m\in K^\times.
	\]
	This defines the isomorphism
	\[
		\delta_K\colon K^\times / (K^\times)^m \xrightarrow{\sim} \operatorname{Hom}(G_{\bar{K}/K} , \mu^m),
	\]
	where~$\delta_K(b)(\sigma) = \beta^\sigma/\beta$ and $\beta$ is chosen such that~$\beta^m=b$.

	With this notation we can define the pairing
	\[
		b\colon \left[E(K)/mE(K)\right] \times E[m] \to K^\times/(K^\times)^m
	\]
	using the following diagram:
	\[
		\xymatrix{
			E[m] \ar[r] &
			\operatorname{Hom}(G_{\bar{K}/K}, E[m]\times E[m]) \ar[r]^-{e_m} &
			\operatorname{Hom}(G_{\bar{K}/K}, \mu_m) \\
			E(K)/mE(K) \ar@{^{(}->}[r]^-{\delta_E} &
			\operatorname{Hom}(G_{\bar{K}/K}, E[m]) \ar[u] &
			K^\times / (K^\times)^m \ar[u]^{\delta_K}_{\sim}, \\
		}
	\]
	where~$e_m(\delta_e(P),T) = \delta_K(b(P,T))$.
	Let us check that the pairing~$b$ is nondegenerate at~$E(K)/mE(K)$.
	Take $P\ne 0 \mod (mE(K))$ in $E(K)$.
	It is taken by~$\delta_E$ to~$h\colon G_{\bar{K}/K}\to E[m]$.
	Since~$h\ne 0$, there exists~$\sigma\in G_{\bar{K}/K}$ such that~$h(\sigma)\ne 0$.
	Since~$e_m$ is nondegenerate, there exists $T\in E[m]$ such that~$e_m(h(\sigma),T)\ne 1 \in \mu_m$.
	Therefore, $\delta_K^{-1}(e_m(h(\sigma),T))\notin (K^\times)^m$.

	Denote by~$S$ the set of all valuations of~$K$ including all non-archimedian, dividing~$m$ and all bad reductions of~$E$.
	Then we can define the group
	\[
		K(S,m) = \{b\in K^\times/(K^\times)^m: \operatorname{ord}_v(b)\,{\vdots}\, m, v\notin S\}.
	\]
	The very useful fact is that the image of the pairing~$b$ lies in~$K(S,m)$.
	We omit the corresponding proof.
	Since~$E[m]\simeq (\mathbb{Z}/m\mathbb{Z})^2$, we can fix there two generators~$T_1$ and~$T_2$.
	Then we obtain the morphism
	\[
		E(K)/mE(K) \to K(S,m) \times K(S,m)
	\]
	defined as
	\[
		P \mapsto (b(T_1,P), b(T_2,P)).
	\]
	Denote~$f_i(P)=b(T_i,P)$, $i=1,2$.
	Since~$b$ is nondegenerate at~$E(K)/mE(K)$ and~$E[m] = \langle T_1,T_2\rangle$, this map is injective.
	The group~$K(S,m)$ can be easily constructed.

	Given fixed~$(b_1,b_2)\in K(S,m)\times K(S,m)$, we are interested whether there exists a point~$P\in E(K)$ such that
	\[
		b_1z_1^m = f_1(P),\quad b_2z_2^m = f_2(P)
	\]
	for some~$z_1,z_2\in K^\times$.
	If we substitute~$P=(x,y)$ to the Weierstra\ss~equation of~$E$, the problem of calculating~$E(K)/mE(K)$ will be reduced to the problem of existence of a solution of
	\[
		\begin{aligned}
			y^2 + a_1xy + a_3y &=& x^3 + a_2x^2 + a_4x + a_6,\\
			b_1z_1^m = f_1(x,y),& &
			b_2z_2^m = f_2(x,y)
		\end{aligned}
	\]
	for~$(x,y,z_1,z_2)\in K\times K\times K^\times\times K^\times$.
	This method is called {\it $m$-descent}.
	From this moment we put~$m=2$.
	The assumption~$E[2]\subset E(K)$ implies that the Weierstra\ss~equation has the form
	\[
		y^2 = (x-e_1)(x-e_2)(x-e_3)\text{ with } e_i\in K.
	\]
	In these coordinates three nontrivial $2$-torsion points are~$T_i=(e_i,0)$.
	It can be checked that if we put~$T=(e,0)$, than the corresponding function~$f$ is~$f(z)=z-e$.
	At~$z=\infty$ and~$z=e$ it can be calculated by linearity.
	Finally, the map~$(f_1,f_2)$ has the form
	\[
		P=(x,y) \mapsto
		\begin{cases}
			(x-e_1,x-e_2), & x\ne e_1,e_2 \\
			\left(\frac{e_1-e_3}{e_1-e_2},e_1-e_2\right), & x=e_1,\\
			\left(e_2-e_1,\frac{e_2-e_3}{e_2-e_1}\right), & x=e_2,\\
			(1,1),& x=\infty.
		\end{cases}
	\]

	In our case the equation has the form
	\[
		y^2 = x^3 + 5x^2 - 32x =
		x
		\left(x + \frac{5 + 3\sqrt{17}}{2} \right)
		\left(x + \frac{5 - 3\sqrt{17}}{2} \right).
	\]
	So we see that in our case~$K=\mathbb{Q}(\sqrt{17})$.
	A point $(b_1,b_2)\in K(S,2)\times K(S,2)$ is the image of a point
	\[
		P = (x,y) \in E(K)/2E(K)
	\]
	not coinciding with~$O$, $(e_1,0)$, $(e_2,0)$ if and only if the system of equations
	\[
		\begin{aligned}
			b_1z_1^2 - b_2z_2^2 = e_2-e_1, \\
			b_1z_1^2 - b_1b_2z_3^2 = e_3-e_1\\
		\end{aligned}
	\]
	has a solution~$(z_1,z_2,z_3)\in K^\times\times K^\times\times K$.
	(We omit the calculation leading to this form.)
	If such a solution exists, we can take
	\[
		P=(x,y) = (b_1z_1^2+e_1,b_1b_2z_1z_2z_3).
	\]
	We put
	\[
		e_1 = 0,\,
		e_2 = -\frac{5 + 3\sqrt{17}}{2},\,
		e_3 = -\frac{5 - 3\sqrt{17}}{2}.
	\]
	Over~$\mathbb{Q}$ the set~$S$ equals~$\{2,3,17,\infty\}$.
	The group of unities of~$\mathbb{Q}(\sqrt{17})$ is $\langle-1\rangle_2\oplus\langle4+\sqrt{17}\rangle_\infty$.
	Therefore the group of unites modulo squares has the set of representatives $\{\pm1,\pm(4+\sqrt{17})\}$.
	Over~$\mathbb{Q}(\sqrt{17})$ the number~$2$ splits as~$\frac{5+\sqrt{17}}2{} \cdot \frac{5-\sqrt{17}}{2}$, $3$ remains prime,~$17$ is the square of~$\sqrt{17}$.
	Denote $\tilde{\imath} = 4+\sqrt{17}$ and~$2_{\pm} = \frac{5 \pm \sqrt{17}}{2}$.
	Denote $\overline{a+\sqrt{17}b} = a - \sqrt{17}b$ the image of the unique nontrivial automorphism of~$\mathbb{Q}(\sqrt{17})$ over~$\mathbb{Q}$.
	Therefore, we can choose the set of representatives of~$K(S,2)$ each of them is a product of some of numbers in the set
	\[
		\left\{
			-1, \tilde\imath, 2_+, 2_-, 3, \sqrt{17}.
		\right\}.
	\]
	So the set of all pair~$(b_1,b_2)$ in~$K(S,2)$ consists of~$(2^6)^2=4096$ pairs.
	We will see below that the system
	\begin{equation}
		\label{eq:b1b2}
		\begin{aligned}
			b_1z_1^2 - b_2 z_2^2 =& e_2, \\
			b_1z_1^2 - b_1b_2 z_3^2 =& e_3
		\end{aligned}
	\end{equation}
	has a solution~$(z_1,z_2,z_3)\in K^\times\times K^\times\times K$ with~$x,y\in\mathbb{Q}$ only for the values~$b_1$ and $b_2$ listed in the following table
	\begin{center}
		\begin{tabular}{|c|c|c|c|c|} \hline
			$b_1$ & $1$ & $-2$ & $-1$ & $2$ \\ \hline
			$b_2$ & $1$ & $\tilde\imath \cdot 2_-$ & $\tilde\imath \cdot 3$ & $2_- \cdot 3$ \\ \hline
		\end{tabular}
	\end{center}

	The rest of the proof is analogous to the proof in example given in~\cite{Silv}, but it is some bit more complicated due to the difficulties arising from the field extension and prime decomposition in its ring of integral elements.
	Proceeding systematically, we list our results in the following table.
	\begin{center}
		\begin{tabular}{|c|c|c|c|c|c|} \hline
			${}_{b_2} {}^{b_1}$ 
			& $1$ 
			& $-1$, $\pm2$ 
			& $\tilde\imath$, $2_\pm$
			& $3$ 
			& $\sqrt{17}$ \\\hline
			$1$ 
			& $O^{\circled{\tiny 1}}$ 
			& \multirow{6}{*}{$\times ^ {\circled{\tiny 7}}$} 
			& \multirow{7}{*}{$\mathbb{Q} ^ {\circled{\tiny 5}}$}
			& \multirow{7}{*}{$\mathbb{Q}_3(i) ^ {\circled{\tiny 4}}$} 
			& \multirow{7}{*}{$\mathbb{Q}_{17}(\sqrt{17}) ^ {\circled{\tiny 3}}$} \\ \cline{1-2}
			$\tilde\imath$ & $\mathbb{Q}_2 ^ {\circled{\tiny 11}}$ &&&& \\ \cline{1-2}
			$3$ & $\mathbb{Q}_3(i) ^ {\circled{\tiny 10}}$ &&&& \\ \cline{1-2}
			$2_+$ & $\mathbb{Q}(\sqrt{17}) ^ {\circled{\tiny 9}}$ &&&& \\ \cline{1-2}
			$\sqrt{17}^*$ & $\mathbb{R} ^ {\circled{\tiny 8}}$ &&&& \\ \cline{1-3}
			$2_-^*$ & \multicolumn{1}{c}{$\mathbb{Q}_2 ^ {\circled{\tiny 6}}$} &&&& \\ \hline
			$-1$ & \multicolumn{4}{c}{$\mathbb{R} ^ {\circled{\tiny 2}}$} & \\ \hline
		\end{tabular}
	\end{center}

	On the first step we find~$z_1,z_2,z_3$ for $b_1$ and~$b_2$ listed in the table above.
	On other steps we consequently exclude possible factors of~$b_1$ and~$b_2$.
	Also on the fifth step we prove that multiplicities of~$2_-$ and~$2_+$ in~$b_1$ are equal, on seventh step we deduce the cases~$b_1 \in \{\pm 1, \pm 2\}$ to the case~$b_1=1$ using the homomorphism property of~$f$.
	It remains to proceed this sequence step by step.

	\begin{itemize}
		\item[\circled{1}]
			Let us find the solution for the cases listed above and the corresponding points on the elliptic curve.
			\begin{enumerate}
				\item 
					Take $(b_1,b_2)=(1,1)$.
					Note that~$\overline{e_2} = e_3$.
					Therefore, if we assume~$z_1\in\mathbb{Q}$ and find~$z_1,z_2$ such that~$z_1^2-z_2^2 = e_2$, then it is enough to take~$z_3=\overline{z_2}$.
					We can take~$z_1=2$, $z_2=\frac{3+\sqrt{17}}{2}$.
					This choice corresponds to the point~$P=(4,4) = 2A$.
				\item
					Take~$(b_1,b_2)=(-2,\tilde\imath\cdot 2_-)$.
					We can take~$z_1=2$, $z_2=\frac{5-\sqrt{17}}{2}$, $z_3=-\frac{1+\sqrt{17}}{4}$ as a solution of
					\[
						\begin{cases}
							-2z_1^2 - \frac{3+\sqrt{17}}{2}z_2^2 = -\frac{5+3\sqrt{17}}{2}, \\
							-2z_1^2 + (3+\sqrt{17})z_3^2 = -\frac{5-3\sqrt{17}}{2}. \\
						\end{cases}
					\]
					This gives the point~$P=(-8,8)=2A+D$.
				\item
					Take~$(b_1,b_2)=(-1,\tilde\imath\cdot 3)$.
					We can take~$z_1=1$, $z_2=\frac{-3+\sqrt{17}}{2}$, $z_3=\frac{5-\sqrt{17}}{4}$ as a solution of
					\[
						\begin{cases}
							-z_1^2 - 3(4 + \sqrt{17})z_2^2 = -\frac{5+3\sqrt{17}}{2}, \\
							-z_1^2 + 3(4+\sqrt{17})z_3^2 = -\frac{5-3\sqrt{17}}{2}. \\
						\end{cases}
					\]
					This gives the point~$P=(-1,6)=3A$.
				\item
					Take~$(b_1,b_2)=(2,2_-\cdot 3)$.
					We can take~$z_1=2$, $z_2=\frac{3¿\sqrt{17}}{2}$, $z_3=-\frac{1+\sqrt{17}}{4}$ as a solution of
					\[
						\begin{cases}
							2z_1^2 - 3\frac{5+\sqrt{17}}{2}z_2^2 = -\frac{5+3\sqrt{17}}{2}, \\
							2z_1^2 + 3(5+\sqrt{17})z_3^2 = -\frac{5-3\sqrt{17}}{2}. \\
						\end{cases}
					\]
					This gives the point~$P=(8,-24)=A+D$.
			\end{enumerate}
		\item[\circled{2}]
			Either~$b_1>0$, or~$b_1<0$:
			\begin{enumerate}
				\item if $b_1 > 0$ and $b_2 < 0$, then the equation $b_1z_1^2 - b_2z_2^2 = e_2 < 0$ has no solution over~$\mathbb{R}$;
				\item if $b_1 < 0$ and $b_2 < 0$, then the equation $b_1z_1^2 - b_1b_2z_3^2 = e_3 > 0$ has no solution over~$\mathbb{R}$.
			\end{enumerate}
			This excludes the factor~$-1$ from~$b_2$, because $\tilde\imath,2_{\pm},3,\sqrt{17}>0$.
		\item[\circled{3}]
			Consider the ramified field extension~$\mathbb{Q}_{17} \subset \mathbb{Q}_{17}(\sqrt{17})$.
			The valuation~$\operatorname{ord}_{17}$ can be continued to the extension and has there half-integer values.
			Here we exclude the factor~$\sqrt{17}$ from~$b_1$ (all others has zero~$\operatorname{ord}_{17}$).
			Assume that~$\operatorname{ord}_{17}b_1=1/2$,\,i.\,e. the factor~$\sqrt{17}$ does contribute to~$b_1$.
			Either~$\operatorname{ord}_{17}b_2=0$, or~$\operatorname{ord}_{17}b_2=1/2$:
			\begin{enumerate}
				\item
					Suppose~$b_1 \operatorname{\vdots} \sqrt{17}$ and $b_2 \operatorname{\not\vdots} \sqrt{17}$.
					The expression~$b_1z_1^2 - b_2z_2^2 = e_2$ has $17$-order~$0$, because~$\operatorname{ord}_{17}\left(-\frac{5+3\sqrt{17}}{2}\right)=0$.
					Since~$\operatorname{ord}_{17}z_i \in \frac{1}{2}\mathbb{Z}$, we have~$\operatorname{ord}_{17}z_i^2\in\mathbb{Z}$, $i=1,2$.
					Therefore $b_1z_1^2$ and~$b_2z_2^2$ has different $17$-orders.
					Their sum has order~$0$.
					Therefore~$z_1$ and~$z_2$ are integral in~$\mathbb{Q}_{17}(\sqrt{17})$.
					Consider the second equation:
					\[
						b_1z_1^2 - b_1b_2z_3^2 = \tilde\imath^{-3} 2_+^5=e_3.
					\]
					Right hand side has $17$-order $0$.
					In the left hand side the first summand has order~$1/2$ and the second summand has half-integer $17$-order.
					Either the last is positive or negative, anyway the equation has no solutions.
				\item
					Suppose~$b_1,b_2 \operatorname{\not\vdots} \sqrt{17}$.
					From
					\[
						b_1z_1^2 - b_1b_2z_3^2 = e_3 = \tilde\imath^{-3} \cdot 2_+^5
					\]
					in the same way it follows that $z_1$ and~$\sqrt{17}z_3$ are integral in~$\mathbb{Q}_{17}(\sqrt{17})$.
					In the equation
					\[
						b_1z_1^2 - b_2z_2^2 = -\tilde\imath^3\cdot 2_-^5
					\]
					the right hand side has order~$0$, in the left hand side the first summand has positive order and the second summand has half-integer order.
					Therefore there are no solutions.
			\end{enumerate}
			Therefore~$\sqrt{17}$ does not contribute to~$b_1$.
		\item[\circled{4}]
			Consider the unramified field extension~$\mathbb{Q}_3 \subset \mathbb{Q}_3(i)$.
			Since~$x^2=17$ has a root in the residue field~$\mathbb{F}_3(i)=\mathbb{F}_9$, by Hensel's lemma it has a root in~$\mathbb{Q}_3(i)$.
			The valuation~$\operatorname{ord}_3$ has integer values on the extension.
			The following reasoning literally repeats the previous case, except the half-integer and integer valuations are replaced with odd and even integer valuations.
			Therefore~$3$ does not contributes to~$b_1$.
		\item[\circled{5}]
			Consider the valuations~$\nu_-$ and~$\nu_+$ at~$2_-$ and~$2_+$.
			They both are even for~$z_1^2$.
			For~$b_1$ they have the values in~$\{0,1\}$, and~$b_1z_1^2 \in\mathbb{Q}$,~i.\,e. they coincide.
			In particular, their parities coincide.
			Therefore, they coincide for~$b_1$.

			It is well known that the integral domain of~$\mathbb{Q}(\sqrt{17})$ is Euclidean and therefore an UFD.
			Any element there is a product of some irreducible elements and a unity.
			The element~$b_1$ is a product of some subset in~$\{-1,2,\tilde\imath\}$.
			Therefore, the element~$b_1z_1^2$ is not a square under the projection into the unity subgroup, because~$\tilde\imath$ is not a square there.

			Hence, $\tilde\imath$ does not contribute to~$b_1$ and~$\operatorname{ord}_{2_-}b_1=\operatorname{ord}_{2_+}b_1$.
			Therefore,~$b_1\in\mathbb{Q}$.
		\item[\circled{6}]
			Here we prove that the prime~$2_-$ has the same degree at~$b_1$ and~$b_2$.
			Consider the field~$\mathbb{Q}_2$.
			The root~$\sqrt{17}$ can be extracted there, because a solution of~$x^2=17$ exists over the residue field~$\mathbb{F}_2$ and a solution over~$\mathbb{Q}_2$ exists by Hensel's lemma.
			(We need to note here that the general version of Hensel's Lemma is not applicable here, and we need to use more general one.
			Actually, there are {\it four} solution of~$x^2=17$ modulo~$2^k$ for~$k\geqslant 3$.
			But there are only {\it two}~$2$-adic limits.)
			Therefore,~$\mathbb{Q}(\sqrt{17})\subset\mathbb{Q}_2$.
			Actually,~$\sqrt{17}$ can be written as~
			\[
				\sqrt{17} = ...10011011101001
			\]
			or
			\[
				\sqrt{17} = ...01100100010111.
			\]
			We fix the embedding corresponding to the first choice of~$\sqrt{17}$ and consider the restriction of the valuation~$\operatorname{ord}_2$ to~$\mathbb{Q}(\sqrt{17})$.
			Note that this valuation on~$\mathbb{Q}(\sqrt{17})$ is not Galois-invariant and depends on the embedding.
			Here~$\sqrt{17} = 5+O(8)$.
			Therefore,
			\[
				\frac{5 + \sqrt{17}}{2} = \frac{14 + O(8)}{2} = 3 + O(4),
			\]
			\[
				\frac{5 - \sqrt{17}}{2} = \frac{-2 + O(8)}{2} = 2 + O(4).
			\]
			In other word,~$\operatorname{ord}_22_-=1$ and~$\operatorname{ord}_22_+=0$.
			Also 
			\[
				e_2=-\frac{5+3\sqrt{17}}{2} = -\frac{5+3(105+O(128))}{2} = 32 + O(64),
			\]
			\[
				e_2=-\frac{5-3\sqrt{17}}{2} = -\frac{5-3(105+O(128))}{2} = 27 + O(64).
			\]

			Suppose the degrees of~$2_-$ in~$b_1$ and~$b_2$ differ.
			Then on of the following cases holds:

			\begin{enumerate}
				\item
					Let~$b_1 \operatorname{\vdots} 2_-$ and~$b_2 \operatorname{\not\vdots} 2_-$.
					Since~$\operatorname{ord}_2 b_1z_1^2$ is odd and $\operatorname{ord}_2b_2z^2$ is even, the equation
					\[
						b_1z_1^2 - b_2z_2^2 = e_2 = 2^5 + O(2^6)
					\]
					implies that~$z_1$ and~$z_2$ are integral in~$\mathbb{Q}_2$, and~$\operatorname{ord}_2z_1=2$.
					In the equation
					\[
						b_1z_1^2 - b_1b_2z_3^2 = e_3
					\]
					we have~$\operatorname{ord}_2 b_1z_1^2 = 5$, and~$\operatorname{ord}_2 b_1b_2z_3^2$ is odd, and~$\operatorname{ord}_2 e_3=0$.
					Therefore, there are no solutions.
				\item
					Let~$b_1 \operatorname{\not\vdots} 2_-$ and~$b_2 \operatorname{\vdots} 2_-$.
					Since~$\operatorname{ord}_2 b_1z_1^2$ is even and $\operatorname{ord}_2b_2z^2$ is odd, the equation
					\[
						b_1z_1^2 - b_2z_2^2 = e_2 = 2^5 + O(2^6)
					\]
					implies that~$z_1$ and~$z_2$ are integral in~$\mathbb{Q}_2$, and~$\operatorname{ord}_2z_1\geqslant 3$.
					In the equation
					\[
						b_1z_1^2 - b_1b_2z_3^2 = e_3
					\]
					we have~$\operatorname{ord}_2 b_1z_1^2 \geqslant 6$, and~$\operatorname{ord}_2 b_1b_2z_3^2$ is odd, and~$\operatorname{ord}_2 e_3=0$.
					Therefore, there are no solutions.
			\end{enumerate}
		\item[\circled{7}]
			Recall that correspondence~$E(\mathbb{Q})/2E(\mathbb{Q}) \to K(S,2)\times K(S,2)$ taking~$P$ to~$(b_1(P),b_2(P))$ is a homomorphism.
			We prove below that for~$b_1=1$ system~\eqref{eq:b1b2} has a solution only for~$b_2=1$.
			Therefore, for each~$b_1$ there is at most one~$b_2$ such that system~\eqref{eq:b1b2} has a solution.
			In \circled{1} we have constructed these solutions.
			Hereafter we consider only~$b_1=1$ and the system
			\begin{equation}
				\label{eq:b1=1}
				\begin{cases}
					z_1^2 - b_2 z_2^2 = e_2, \\
					z_1^2 - b_2 z_3^2 = e_3 = \overline{e_2}. \\
				\end{cases}
			\end{equation}
		\item[\circled{8}]
			We have~$x = z_1^2 \in\mathbb{Q}$.
			Therefore~$\overline{b_2z_2^2} = b_2z_3^2$ and~${z_3}/\overline{z_2} = \pm \sqrt{\overline{b_2}/b_2} \in \mathbb{Q}(\sqrt{17})$.
			If~$\overline{b_2}/b_2$ is not a square in~$\mathbb{Q}(\sqrt{17})$, then there are no solutions of~\eqref{eq:b1=1}.
			The expression~$\overline{b_2}/b_2$ is multiplicative in~$b_2$ and is positive for~$b_2=1,3,2_+$ and is negative for~$b_2=\tilde\imath,\sqrt{17}$.
			Since~$\mathbb{Q}(\sqrt{17}) \subset \mathbb{R}$, the degrees of~$\sqrt{17}$ and~$\tilde\imath$ in~$b_2$ coincide.
		\item[\circled{9}]
			Consider again the expression~$B=\overline{b_2}/b_2$ that must be an element of~$\mathbb{Q}{\sqrt{17}}$, if there are a solution of~\eqref{eq:b1=1}.
			Multiplication of~$b_2$ with~$3$ does not change~$B$, multiplication with~$\tilde\imath$ divides~$B$ by~$\tilde\imath^2$.
			Therefore, we need only to consider~$b_2=2_+$ and~$b_2=2_+\sqrt{17}$ to exclude the factor~$2_+$ from~$b_2$.
			Take~$b_2=2_+$.
			One can see that the minimal polynomial of~$B=\sqrt{\overline{b_2}/b_2} = \sqrt{\frac{5-\sqrt{17}}{5+\sqrt{17}}}$ is~$2x^4-21x^2+2$.
			Therefore $[\mathbb{Q}(B):\mathbb{Q}] = 4$, and~$B$ can not lie in any quadratic extension of~$\mathbb{Q}$, in particular,~$\mathbb{Q}(\sqrt{17})$.

			For~$b_2=2_+\sqrt{17}$ we get the minimal polynomial~$2x^2+21x^2+2$.
			In the same way $B\not\in \mathbb{Q}(\sqrt{17})$.
		\item[\circled{10}]
			It remains to consider~$b_2=3^{\varepsilon_1}\tilde\imath^{\varepsilon_2}$, where~$\varepsilon_1,\varepsilon_2=0,1$.
			If a pair~$(z_1,z_2)$ is a solution of~$z_1^2-b_2z_2^2=e_2$, then the triple~$(z_1,z_2,z_3=\tilde\imath^{-\varepsilon_2}\overline{z_2})$ is a solution of~\eqref{eq:b1=1}.
			Therefore $b_2z_2z_3 = 3^{\varepsilon_1}(\tilde\imath\sqrt{17})^{\varepsilon_2}z_2(-\tilde\imath\sqrt{17})^{-\varepsilon_2}\overline{z_2} = (-1)^{\varepsilon_2}3^{\varepsilon_1}z_2\overline{z_2} \in \mathbb{Q}$.
			In the same time~$y=b_2z_1z_2z_3\in\mathbb{Q}$.
			This implies that~$z_1\in\mathbb{Q}$.

			Since~$x^2=17$ has a non-zero root in~$\mathbb{F}_3(i)=\mathbb{F}_9$, we have an embedding~$\mathbb{Q}(\sqrt{17}) \subset \mathbb{Q}_3(i)$.
			The valuation~$\operatorname{ord}_3$ can be extended from~$\mathbb{Q}_3$ to~$\mathbb{Q}_3(i)$ and has there integer values.
			We have~$z_1^2 - 3(\tilde\imath\sqrt{17})^{\varepsilon_2}z_2^2 = e_2$.
			Here~$\operatorname{ord}_3e_2=\operatorname{ord}_3\tilde\imath=\operatorname{ord}_3\sqrt{17}=0$.
			Therefore in the residue field~$\mathbb{F}_9$ we obtain~$z_1^2\equiv e_2$.

			Denote~$\sqrt{17}=i$ in~$\mathbb{F}_9$.
			Then~$e_2=-\frac{5-3\sqrt{17}}{2} = -\frac{5-3i}{2} = 2+O(3)$ in~$\mathbb{Q}_3(i)$.
			As above,~$z_1$ and~$z_2$ are integral in~$\mathbb{Q}_3(i)$.
			Therefore,~$z_1^2 \equiv 2$ in~$\mathbb{F}_9$.
			But it was proved above that~$z_1\in\mathbb{Q}$.
			This means that its square is not~$2$ modulo~$3$.
		\item[\circled{11}]
			It only remains to check the case~$b_2 = \tilde\imath\sqrt{17}$.
			Existence of solution~$(z_1,z_2)\subset K^\times\times K^\times$ is equivalent to existence of a point defined over $\mathbb{Q}(\sqrt{17})$ on the plane conic
			\[
				z_1^2 - \tilde\imath\sqrt{17}z_2^2 = e_2 = -\tilde\imath^32_-^5.
			\]
			At first, we can projectivize this conic multiplying the right hand side with~$\tilde\imath^{-2}2_-^{-4}z_4^2$.
			Therefore, we can consider the equation
			\[
				\tilde\imath z_5^2 + 2_-z_4^2 = \sqrt{17},
			\]
			where~$z_5 = \tilde\imath z_1$.
			As usual, we obtain that~$z_4$ and~$z_5$ are integral in~$\mathbb{Q}_2$ if we take the embedding~$\mathbb{Q}(\sqrt{17})$ mentioned above.
			We have~$\sqrt{17} = 1 + O(8)$ and~$\tilde\imath = 1 + O(8)$.
			Therefore,~$\operatorname{ord}_2z_4 > 0$, and~$\operatorname{ord}_2 2_-z_4^2 \geqslant 3$.

			Consequently, $z_5^2 = \frac{1 + O(8)}{5 + O(8)} = 5 + O(8)$, where~$z_5$ is an integral~$2$-adic numbed.
			Taking it modulo~$8$, we get~$z_1^2 \equiv 5 \mod 8$.
			But~$5$ is not a square in~$\mathbb{Z}/8\mathbb{Z}$.
			Therefore the equation has no solutions.
	\end{itemize}

	\begin{conclusion}
		$\mathcal{E}(\mathbb{Q}) = \mathbb{Z} \oplus\frac{\mathbb{Z}}{2\mathbb{Z}}$.
	\end{conclusion}

\appendix
	\section{Examples}
	The curve~$\mathcal{E}$ can be written in some different coordinates:
	\begin{enumerate}
		\item $y^2+y = x^3+x^2-2x$ (minimal form),
		\item $y^2=x^3+5x^2-32$ (the form we use in $2$-descent algorithm),
		\item $c^3+c^2(a+b)=c(a^2+ab+b^2)+a^3+a^2b+ab^2+b^3$ (initial form).
	\end{enumerate}

	In the following table some points of~$\mathcal{E}(\mathbb{Q})$ are listed:

	{
	\scriptsize
	\begin{tabular}{|c|l|l|l|l|}\hline
		pt & 1 & 2 & 3 \\ \hline\hline
		$O$ & \multicolumn{2}{c|}{$(0:1:0)$} & $(1:-1:0)$ \\ \hline
		$D$ & \multicolumn{2}{c|}{$(0:0:1)$} & $(-1:-1:1)$ \\ \hline
		$A$ & $(-1:-1:1)$ & $(-4:-12:1)$ & $(1:0:-1)$ \\ \hline
		$A+D$ & $(2:-4:1)$ & $(8:-24:1)$ & $(1:-1:1)$ \\ \hline
		$2A$ & $(1:0:1)$ & $(4:4:1)$ & $(0:1:1)$ \\ \hline
		$2A+D$ & $(-2:2:1)$ & $(-8:8:1)$ & $(1:3:-5)$ \\ \hline
		$3A$ & $(-2:7:8)$ & $(-1:6:1)$ & $(1:5:-4)$ \\ \hline
		$3A+D$ & $(8:20:1)$ & $(32:192:1)$ & $(-9:7:5)$ \\ \hline
		$4A$ & $(9:-33:1)$ & $(36:-228:1)$ & $(25:-32:17)$ \\ \hline
		$4A+D$ & $(-6:-16:27)$ & $(-24:-152:27)$ & $(49:11:-39)$ \\ \hline
		$5A$ & $(-245:-14:125)$ & $(-980:-1092:125)$ & $(128:37:-205)$ \\ \hline
		$5A+D$ & $(350:-370:343)$ & $(1400:-1560:343)$ & $(121:-9:119)$ \\ \hline
		$6A$ & $(968:913:512)$ & $(484:1397:64)$ & $(-1369:1425:1424)$ \\ \hline
		$6A+D$ & $(-1408:2736:1331)$ & $(-5632:16256:1331)$ & $(3:4067:-4147)$ \\ \hline
		$7A$ & $(-37:1955:50653)$ & $(-148:15492:50653)$ & $(16245:17536:-16909)$ \\ \hline
		$7A+D$ & $(2738:141932:1)$ & $(10952:1146408:1)$ & $(-48223:47311:1825)$ \\ \hline
		$8A$ & $(392673:-808088:185193)$ & $(1570692:-4894012:185193)$ & $(600608:-622895:600153)$ \\ \hline
		$8A+D$ & $(-539334:-570570:571787)$ & $(-2157336:-6721896:571787)$ & $(1681691:1217:-1650455)$ \\ \hline
		$9A$ & $(-41889394:39480443:20570824)$ & $(-20944697:18535746:2571353)$ & $(7659925:20017089:-34783204)$ \\ \hline
		$9A+D$ & $(58709432:-3376228:59776471)$ & $(234837728:207827904:59776471)$ & $(1481089:18800081:19214131)$ \\ \hline
	\end{tabular}
	}

	Recall that for given point~$Z=(x:y:z)$ the point~$-Z$ is~$(x:z-y:z)$ in the first case and is~$(x:-y:z)$ in the second case, for the point~$Z=(a:b:c)$ in the third case the point~$-Z$ is~$(b:a:c)$.

	Let us find some of integer triangles.
	Minimal among them are the following.

	\begin{itemize}
		\item
		given in the introduction triangle corresponding to~$9A+D$:
		\[
			(83^2\cdot2729, 1217^2, 17\cdot23\cdot157\cdot313).
		\]

		\item 
		the triangle corresponding to~$16A$:
		\begin{multline*}
			(2^5 \cdot 5^2 \cdot 17 \cdot 23 \cdot 137 \cdot 7901 \cdot 943429^2, \\
			29^2 \cdot 37 \cdot 1291 \cdot 3041^2 \cdot 11497^2, \\
			3 \cdot 19 \cdot 83 \cdot 2593 \cdot 14741 \cdot 227257 \cdot 7704617).
		\end{multline*}

		\item
		the triangle corresponding to~$23A+D$:
		\begin{multline*}
			({5}\cdot {17}\cdot {29}\cdot {97}\cdot {17182729}\cdot {32537017}\cdot {254398174040897}\cdot {350987274396527},\\
			{7}\cdot {1093889^2}\cdot {4941193}\cdot {894993889^2}\cdot {331123185233},\\
			{83^2}\cdot {571^2}\cdot {13873}\cdot {337789537}\cdot {16268766383521^2}).
		\end{multline*}

		\item
		the triangle corresponding to~$30A$:
		\begin{multline*}
			\hspace*{-5mm}
			(
			3 \cdot 19 \cdot 83 \cdot 347 \cdot 853^2 \cdot 14741 \cdot 197609 \cdot 1326053 \cdot 9921337^2 \cdot 2774248223 \cdot 16439698126501721^2, \\
			37 \cdot 53 \cdot 113 \cdot 1291 \cdot 6301^2 \cdot 11057 \cdot 70717^2 \cdot 419401^2 \cdot 56702749^2 \cdot 75758233^2 \cdot 58963203163, \\
			2^4 \cdot 5 \cdot 7 \cdot 13 \cdot 281 \cdot 1361 \cdot 4519 \cdot 943429 \cdot 1277496791 \cdot 58636722172129 \times\\\times 434222192069971469300337687991080717947321).
		\end{multline*}

		\item
		the triangle corresponding to $37A+D$:
		\begin{multline*}
			(
			5\cdot 2299159\cdot 138049208211121\cdot 2760426916410799\cdot 728165182513369014929\times\\\times 2457244522753608004147669717\cdot 3646312514774768838959262707271994342627321, \\
			{3^6}\cdot {41}\cdot {43^2}\cdot {59^2}\cdot {71}\cdot {1753^2}\cdot {4271}\cdot {6449^2}\cdot {306193^2}\cdot {258408497^2}\times\\\times {294583400141651^2}\cdot {5917115594031382979839359182507437287191},\\
			7^2\cdot 79\cdot 3529\cdot 2812999081^2\cdot 5544800297^2\cdot 16078869119\times\\\times 13860847191174419174377^2\cdot 306179686612030303942777
			).
		\end{multline*}

		\item
		the following triangles correspond to $44A$, $51A+D$, $58A$, $65A+D$, $72A$, $79A+D$, $86A$, $93A+D$, $100A$, $107A+D$, $114A$, $121A+D$, $132A$, ...
	\end{itemize}

\end{document}